\documentclass[11pt,a4paper]{scrartcl}
\usepackage[utf8]{inputenc}
\usepackage{amsmath}
\usepackage{amsfonts}
\usepackage{amssymb}
\usepackage{amsthm}
\usepackage{graphicx}
\usepackage{enumitem}
\usepackage{lmodern}
\usepackage{microtype}
\usepackage[]{scrlayer-scrpage}
\usepackage[english]{babel}
\usepackage{xspace}		
\usepackage{setspace}   
\usepackage[autostyle]{csquotes} 
\usepackage{xpatch} 
\usepackage{aliascnt} 

\usepackage[
style=numeric,    		
isbn=false,                
pagetracker=true,          
maxbibnames=5,            
maxcitenames=3,            
giveninits=true,			
autocite=inline,           
block=space,               
date=comp,                
dateabbrev=false,			
backend=biber,
]{biblatex}
\DeclareFieldFormat[article,incollection,inproceedings,thesis]{citetitle}{#1\midsentence} 
\DeclareFieldFormat[article,incollection,inproceedings,thesis]{title}{#1\midsentence}
\renewrobustcmd*{\bibinitdelim}{\,} 
\xpatchbibdriver{online}{\usebibmacro{date}}{}{}{}
\xpatchbibdriver{online}{ \usebibmacro{addendum+pubstate}}{\addcomma\space\usebibmacro{publisher+location+date}\newunit\newblock\usebibmacro{addendum+pubstate}}{}{} 
\bibliography{bib}   
\usepackage{url} 
\usepackage[hidelinks]{hyperref} 
\usepackage[capitalise,nameinlink,noabbrev]{cleveref}
\crefformat{equation}{#2(#1)#3}

\theoremstyle{definition}
\newtheorem{definition}{Definition}

\theoremstyle{plain}
\newtheorem{theorem}{Theorem}[section]
\newaliascnt{lemma}{theorem}
\newtheorem{lemma}[lemma]{Lemma}
\aliascntresetthe{lemma}
\newaliascnt{corollary}{theorem}
\newtheorem{corollary}[corollary]{Corollary}
\aliascntresetthe{corollary}
\newaliascnt{proposition}{theorem}
\newtheorem{proposition}[proposition]{Proposition}
\aliascntresetthe{proposition}

\theoremstyle{remark}
\newtheorem{remark}{Remark}

\newcommand{\F}{\mathbb{F}}
\newcommand{\Fq}{\mathbb{F}_q}
\newcommand{\Fqfour}{\mathbb{F}_{p^{4n}}}
\newcommand{\GR}{\text{GR}}
\newcommand{\Z}{\mathbb{Z}}

\newcommand{\T}{\mathcal{T}}
\newcommand{\I}{\mathcal{I}}

\begin{document}
	\title{Solving isomorphism problems about 2-designs from disjoint difference families}
	\author{Christian Kaspers\thanks{Faculty of Mathematics, Otto von Guericke University Magdeburg, 39106 Magdeburg, Germany (email: \href{mailto:christian.kaspers@ovgu.de}{\nolinkurl{christian.kaspers@ovgu.de}}, \href{mailto:alexander.pott@ovgu.de}{\nolinkurl{alexander.pott@ovgu.de}})} \space and Alexander Pott$^*$}
	\maketitle
	
	\begin{abstract}
		Recently, two new constructions of $(v,k,k-1)$ disjoint difference families in Galois rings were presented by \textcite{davis2017,momihara2017}. Both were motivated by a well-known construction of difference families from cyclotomy in finite fields by \textcite{wilson1972}. It is obvious that the difference families in the Galois ring and the difference families in the finite field are not equivalent. A related question, which is in general harder to answer, is whether the associated designs are isomorphic or not. In our case, this problem was raised by the authors of \cite{davis2017}. In this paper we show that the $2$-$(v,k,k-1)$ designs arising from the difference families in Galois rings~\cite{davis2017,momihara2017} and those arising from the difference families in finite fields~\cite{wilson1972} are nonisomorphic by comparing their block intersection numbers.
	\end{abstract}
	
	
\section{Introduction}
	\label{sec:introduction}
	Various types of difference families have long been studied in combinatorial literature \cite{abel2006,beth1999,chang2006,furino1991,wilson1972}. They have applications in coding theory, and communications and information security \cite{ng2016}, and they have connections to many other combinatorial objects, in particular to combinatorial designs. When, like recently by \textcite{davis2017} and by \textcite{momihara2017}, a new construction of difference families is presented, it is a natural question to ask whether the new construction also leads to new difference families. In this paper, we compare two infinite families of difference families in Galois rings \cite{davis2017,momihara2017} with a well-known infinite family of difference families in finite fields which was introduced by \textcite{wilson1972} in 1972. Since the difference families we compare are in different groups, they cannot be equivalent. Hence, we will examine whether the associated combinatorial designs are isomorphic or not. This question is of particular interest because the authors of both \cite{davis2017} and \cite{momihara2017} mention that they were inspired by the construction of \textcite{wilson1972}, and the authors of \cite{davis2017} state explicitly that it is not known in general if the associated designs are nonisomorphic.\par
	We start by defining the relevant objects we examine in this paper.	First, we need the following notations: Let $G$ be an abelian group, $A,B \subseteq G$ and $g \in G$. We define multisets
	\begin{align*}
		\Delta A 		&:= \{ a - a' : a, a' \in A, a \ne a'\},\\
		\Delta_+ A 		&:= \{ a + a' : a, a' \in A, a \ne -a'\},\\
		A-B			 	&:= \{a-b : a \in A, b \in B, a \ne b\},\\
		A + B 			&:= \{a+b: a \in A, b \in B, a \ne -b\},\\
		A + g 			&:= \{a + g : a \in A\}.
	\end{align*}
	In the course of this paper we will sometimes use these notations to denote sets, not multisets. It will be clear from the context if the multiset or the respective set is meant.
	
	\begin{definition}
	\label{def:DF}
		Let $G$ be an abelian group of order $v$, and let $D_1, D_2, \dots, D_b$ be $k$-subsets of $G$. The collection $D = \{D_1, D_2, \dots, D_b\}$ of these subsets is called a \emph{difference family in G with parameters $(v,k,\lambda)$} if each nonzero element of $G$ occurs exactly $\lambda$ times in the multiset union
		\[
			\bigcup_{i=1}^b \Delta D_i.
		\]
		If the subsets $D_1, D_2, \dots, D_b$ are mutually disjoint, they form a \emph{disjoint difference family}. If $b = 1$, one speaks of a \emph{$(v,k,\lambda)$ difference set}. We call $D$ \emph{complete} (or \emph{near-complete}) if the $D_i$ partition $G$ (or $G \setminus \{0\}$, respectively).
	\end{definition}
	
	In this paper we will focus on near-complete $(v,k,k-1)$ disjoint difference families. These objects are closely related to so-called external difference families: 
	
	\begin{definition}
	\label{def:EDF}
		Let $G$ be an abelian group of order $v$, and let $D_1, D_2, \dots, D_b$ be mutually disjoint $k$-subsets of $G$. The collection $D = \{D_1, D_2, \dots, D_b\}$ of these subsets is called a \emph{$(v,k,\lambda)$ external difference family} if each nonzero element of $G$ occurs exactly $\lambda$ times in the multiset union
		\[
			\bigcup_{\substack{1 \le i,j \le b \\ i \ne j}} \left(D_i - D_j\right).
		\]
		Analogously to \autoref{def:DF} we call an external difference family \emph{complete} (or \emph{near-complete}) if the $D_i$ partition the (nonzero) elements of $G$.
	\end{definition}
	
	The following \autoref{prop:EDF_DDF} shows that under certain conditions a disjoint difference family is also an external difference family. This result was observed by \textcite{momihara2017}, and, for near-complete disjoint difference families, it was also mentioned by \textcite{chang2006} and  \textcite{davis2017}. We will add a short proof.

	\begin{proposition}
	\label{prop:EDF_DDF}
		Let $G$ be an abelian group of order $v$, and let $D = \{D_1, D_2, \dots, D_b\}$ be a $(v,k,\lambda)$ disjoint difference family in $G$. The collection $D$ forms an external difference family in $G$ if and only if the union $\bigcup_{i=1}^{b} D_i$ of the $D_i$ is a $(v,bk,\lambda')$ difference set in $G$ for some integer $\lambda'>\lambda$. As an external difference family, $D$ has parameters $(v,k,\lambda'-\lambda)$.
	\end{proposition}
	\begin{proof}
		Let $G$ be an abelian group of order $v$, and let $D = \{D_1, D_2, \dots, D_b\}$ be a collection of mutually disjoint $k$-subsets whose union $\bigcup_{i=1}^{b} D_i$ is a $(v,bk,\lambda')$ difference set for some integer $\lambda'>2$. We can split all the differences in $\bigcup_{i=1}^{b} D_i$ in the following way into the \enquote{internal} and the \enquote{external} differences of the $D_i$:
		\[
			\Delta \left(\bigcup_{i=1}^{b} D_i\right) = \bigcup_{i=1}^b \Delta D_i \cup \bigcup_{\substack{1 \le i,j \le b \\ i \ne j}} (D_i-D_j).
		\]
		Since $\bigcup_{i=1}^{b} D_i$ is a difference set, each element $g \in G \setminus \{0\}$ is represented as $\lambda'$ differences in $\Delta(\bigcup_{i=1}^{b} D_i)$. It follows that each nonzero element in $G$ is represented as $\lambda$ differences in $\bigcup_{i=1}^b \Delta D_i$ (meaning $D$ is a $(v,k,\lambda)$ disjoint difference family) if and only if it is represented $\lambda'-\lambda$ times in $\bigcup_{1 \le i,j \le b,\, i \ne j} (D_i-D_j)$ (meaning $D$ is a $(v,k,\lambda'-\lambda)$ external difference family).
	\end{proof}
	From \autoref{prop:EDF_DDF} it follows that every near-complete $(v,k,k-1)$ disjoint difference family is also a near-complete $(v,k,v-k-1)$ external difference family because the $D_i$ partition $G\setminus\{0\}$ and $G\setminus\{0\}$ is a $(v,v-1,v-2)$ difference set in $G$.
%
	For extended background on $(v,k,k-1)$ disjoint difference families the reader is referred to \textcite{buratti2017} who gives an overview over these difference families and summarizes several constructions, including the one by \textcite{davis2017} (see \autoref{sec:EDF_GR} for some additional information). As mentioned above, every difference family gives rise to a combinatorial design.
	\begin{definition}
	\label{def:design}
		Let $P$ be a set with $v$ elements (\emph{points}). A \emph{$t$-$(v,k,\lambda)$~design} (or \emph{$t$-design}, in brief) is a collection of $k$-subsets (\emph{blocks}) of $P$ such that each $t$-subset of $P$ is contained in exactly $\lambda$ blocks.
	\end{definition}
	 The designs coming from difference families are $2$-designs, which are often referred to as \emph{balanced incomplete block designs (BIBD)}. They can be constructed from difference families in the following way.
	
	\begin{definition}
		Let $G$ be an abelian group, and let $D = \{D_1, D_2, \dots, D_b\}$ be a family of subsets of $G$. The \emph{development $dev(D)$ of $D$} is the collection
		\[
			\left\lbrace D_i + g : D_i \in D, g \in G\right\rbrace
		\]
		of all the translates of the subsets contained in $D$. The sets $D_1, D_2, \dots, D_b$ are called the \emph{base blocks} of $dev(D)$.
	\end{definition}
	
	In other words: The development $dev(D)$ of $D$ contains the orbits of the sets $D_i \in D$ under the action of $G$. If all the orbits have full length, $dev(D)$ consists of $vb$ blocks. The following \autoref{prop:dev_design} is well known. We will add a proof for completeness.
	
	\begin{proposition}
	\label{prop:dev_design}
		Led $D$ be a $(v,k,\lambda)$ difference family in an abelian group $G$. The development $dev(D)$ of $D$ forms a $2$-$(v,k,\lambda)$ design with point set $G$.
	\end{proposition}
	\begin{proof}
		Let $D = \{D_1, D_2, \dots, D_b\}$ be a $(v,k,\lambda)$ difference family in an abelian group $G$. Take an arbitrary $2$-subset $\{t_1, t_2\}$ of $G$. We need to show that $\{t_1, t_2\}$ is contained in $\lambda$~blocks of $dev(D)$. Let $d = t_1-t_2$. Since $d \ne 0$, $d$ is represented $\lambda$ times as a difference $d = d' - d''$, where $d', d'' \in D_i, 1\le i \le b$. Obviously, the differences in a base block $D_i$ and in all its translates $D_i+g, g \in G$, are the same, in short: $\Delta D_i = \Delta (D_i + g)$ for all $1 \le i \le b$ and $g \in G$. Hence, for each of the $\lambda$ pairs $d', d''$ we choose $g$ such that $d' + g = t_1$. Then $d'' + g = t_2$ and consequently $\{t_1, t_2\} \subseteq D_i + g$.
	\end{proof}
	
\section{Disjoint difference families from cyclotomy in finite fields}
	\label{sec:DDF_Fq}
	First, we present the well-known construction of disjoint difference families in finite fields by \textcite{wilson1972}. It makes use of the cyclotomoy of the $e$-th powers in a finite field. Let $q$ be a power of a prime $p$. We denote by $\Fq$ the finite field with $q$ elements and by $\alpha$ a generator of the multiplicative group $\Fq^*$ of $\Fq$. 
	
	\begin{theorem}
	\label{th:Fq_DDF}
		Let $e,f$ be integers satisfying $ef = q-1$, $e,f \geq 2$, and let 
		\begin{align*}
		\label{eq:C_i}
			C_i = \{ \alpha^{t}\ | \ t \equiv i \pmod e  \}, i = 0, 1, \dots, e-1,
		\end{align*}
		be the cosets of the unique subgroup $C_0$ of index $e$ and order $f$ formed by the $e$-th powers of~$\alpha$ in $\Fq^*$. Then, the family $C = \{C_0, C_1, \dots, C_{e-1}\}$ of all these cosets forms a $(q, f, f-1)$ near-complete disjoint difference family in the additive group $(\Fq,+)$.
	\end{theorem}
	\begin{proof}
		This proof is similar to the proof by \textcite[Theorem~2.1]{davis2017} where the authors prove that $C$ is a near-complete external difference family (see \autoref{prop:EDF_DDF}). Let $x,y \in \Fq^*$, and let $z = yx^{-1}$. Suppose $x = c - c'$ for $c, c' \in C_i, 0 \le i \le e-1$. Then $y = zx = zc - zc'$ and, obviously, $zc, zc'$ are in the same coset $C_j$. Hence, we have found a representation of $y$ as the difference of two distinct elements from the same set $C_j$. The other way around, every difference for $y$ will give us a difference for $x$. Consequently, every element of $\Fq^*$ will have the same number of differences. So, all we need to do is to count the total number of differences and divide it by the number of elements in $\Fq^*$: There are $e$~sets~$C_i$, and in each $C_i$ we can calculate $f(f-1)$ differences, giving us a total of $ef(f-1)$ differences. In $\Fq^*$ there are $q-1 =ef$ elements. Hence, each element $x \in \Fq^*$ will have
		\[
			\frac{ef(f-1)}{ef} = f-1
		\]
		differences $x = c -c'$ where $c$ and $c'$ come from the same set $C_i$.
	\end{proof}
	
	We remark that this theorem can also be proved using the results of \textcite[Theorem~3.3, Corollary~3.5]{furino1991}. Employing the construction of a $2$-design mentioned above, the development $dev(C)$ of $C$ is a $2$-$(q, f, f-1)$~design.
	
	\section{Galois rings}
	\label{sec:Galois_rings}
	In this section, we will give a short introduction to Galois rings, see the work by \textcite{wan2003} for extended general background on this topic. Let $p$ be a prime, and let $f(x) \in \Z_{p^m}[x]$ be a monic basic irreducible polynomial of degree~$r$. The factor ring $\Z_{p^m}[x]/ \langle f(x)\rangle$ is called a \emph{Galois ring} of characteristic $p^m$ and extension degree $r$. It is denoted by $\GR(p^m, r)$, and its order is $p^{mr}$. Since any two Galois rings of the same characteristic and the same order are isomorphic, we will speak of \emph{the} Galois ring $\GR(p^m, r)$.\par
	Galois rings are local commutative rings. The unique maximal ideal of the ring $R := \GR(p^m, r)$ is $\I = pR := \{pa : a \in R\}$. The factor ring $R / \I$ is isomorphic to the finite field $\F_{p^r}$ with $p^r$ elements. As a system of representatives of $R/\I$ we take the \emph{Teichmüller set} $\T = \{0, 1, \xi, \dots, \xi^{p^r-2}\}$ where $\xi$ denotes a root of order $p^r-1$ of $f(x)$. It is convenient to choose the \emph{generalized Conway polynomial}, i.\,e. the Hensel lift from $\F_p[x]$ to $\mathbb{Z}_{p^m}[x]$ of the Conway polynomial, as our polynomial $f(x)$ since then, $x +(f)$ is a generator of the Teichmüller group, and we set $\xi = x + (f)$ (see \autocite[Section~1.3]{zwanzger2011} for more information on the generalized Conway polynomial and its construction). An arbitrary element $a$ of $R$ has a unique \emph{$p$-adic representation} $a = \alpha_0 + p\alpha_1 + \dots + p^{m-1}\alpha_{m-1}$, where $\alpha_0, \alpha_1, \dots, \alpha_{m-1} \in \T$.\par
	The elements of $R \setminus \I$ are all the units of $R$, this \emph{unit group} is denoted by $R^*$. It has order $p^{mr} - p^{(m-1)r} = p^{(m-1)r}(p^r-1)$ and is the direct product of the cyclic \emph{Teichmüller group $\T^* = \T \setminus \{0\} = \{ 1, \xi, \dots, \xi^{p^r-2}\}$} of order $p^r-1$ and the \emph{group of principal units $\mathbb{P} := 1 + \I$} of order $p^{(m-1)r}$. If $p$ is odd or if $p = 2$ and $m \leq 2$, then $\mathbb{P}$ is a direct product of $r$ cyclic groups of order $p^{m-1}$. If $p = 2$ and $m \geq 3$, then $\mathbb{P}$ is a direct product of a cyclic group of order $2$, a cyclic group of order $2^{m-2}$ and $r-1$ cyclic groups of order~$2^{m-1}$. So we have $R^* = \T^* \times \mathbb{P}$. In this paper, we will only consider Galois rings of characteristic~$p^2$. In this case, $(1+p\alpha)(1+p\beta) = 1+p(\alpha + \beta)$ for any $\alpha, \beta \in \T$, and every unit $u \in \GR(p^2, r)^*$ has a unique representation $u = \alpha_0 (1 + p\alpha_1), \alpha_0, \alpha_1 \in \T, \alpha_0 \ne 0$. Moreover, if $m=2$, the group of principal units $\mathbb{P}$ is a direct product of $r$ cyclic groups of order $p$ and thus has the structure of an elementary abelian group of order $p^r$.
	
\section{Disjoint difference families in Galois rings I}
	\label{sec:DDF_GR}
	We will now present the new construction of disjoint difference families in Galois rings $\GR(p^2, 2n)$ with characteristic $p^2$ and even degree $r = 2n, n \in \mathbb{N}$, that was introduced by \textcite{momihara2017}: 
	Let $p$ be a prime. Let $R_{2n}$ denote the Galois ring $\GR(p^2, 2n) = \Z_{p^2}[x] / \langle f(x) \rangle$, where $f(x)$ is a monic basic irreducible polynomial of degree $2n$, and let $\xi$ be a root of order $p^{2n}-1$ of $f(x)$. Let $\I_{2n}$ be the maximal ideal and let $\mathbb{P}_{2n} = 1 + \I_{2n}$ be the group of principal units of $R_{2n}$. Moreover, we have the Teichmüller set 
	\[
		\T_{2n} = \{ 0, 1, \xi, \dots, \xi^{p^{2n}-2}\},
	\]
	and each element of $R_{2n}$ has a unique $p$-adic representation
	\[
		a_0 + pa_1, \quad a_0, a_1 \in \T_{2n}.
	\]
	The Galois ring $R_{2n}$ contains a unique Galois ring $R_n = \GR(p^2, n)$ of characteristic $p^2$ and degree $n$ as its subring \autocite[Theorem~14.24]{wan2003}. It can be constructed in the following way \autocite[Corollary~14.28]{wan2003}: Obviously, $\xi^{(p^n+1)}$ is a root of order $p^n-1$ of $f(x)$. It follows that
	\[
		\T_n = \{ 0, 1, \xi^{p^n+1}, \xi^{2(p^n+1)}, \dots, \xi^{(p^n-2)(p^n+1)}\}
	\]
	is the Teichmüller set of $R_n$ and $R_n = \{a_0 + pa_1: a_0, a_1 \in \T_n \}$. Then,
	\[
		R_n^* = \{ \alpha_0(1+p\alpha_1) :\alpha_0,\alpha_1 \in \T_n, \alpha_0 \ne 0\}
	\]
	is the unit group of the subring $R_n$. Analogously to $R_{2n}$, let $\I_n$ denote the maximal ideal and $\mathbb{P}_n$ the group of principal units of $R_n$. So we have $R_n^* = \T_n^* \times \mathbb{P}_n$. Now, let $pS$ be a system of representatives of $\I_{2n}/\I_n$ which means that $1 + pS$ will be a system of representatives of $\mathbb{P}_{2n} / \mathbb{P}_n$. Each element of $1+ pS$ can be written as $1 + px$ for some $x \in \T_{2n}$, i.\,e. $1 + pS = \{1+px : x \in S\}$ for some subset $S$ of $\T_{2n}$. Write $S = \{x_0, x_1, \dots, x_{p^n-1}\}$.	Finally, define a coset $P$ of the maximal ideal $\I_n$ of $R_n$ as
	\[
		P = \{ p\xi^{p^n}, p\xi^{(p^n+1)+p^n}, p\xi^{2(p^n+1)+p^n}, \dots, p\xi^{(p^n-2)(p^n+1)+p^n} \}.
	\]
	\begin{theorem}[{\autocite[Theorem~1]{momihara2017}}]
	\label{th:momihara_3.1}
		Using the notation from above, define subsets
		\begin{equation*}
			D_i = \xi^i \left( P \cup \left( \bigcup_{j = 0}^{p^n-1} \xi^j (1+px_j)R_n^*\right)\right), i = 0, 1, \dots, p^n,
		\end{equation*}
		of $R_{2n}$. The family $D = \{D_0, D_1, \dots, D_{p^n}\}$ forms a near-complete disjoint difference family in $(R_{2n}, +)$ with parameters $\left(p^{4n}, (p^{2n}+1)(p^n-1), (p^{2n}+1)(p^n-1) - 1 \right)$.
	\end{theorem}
	For the extensive and technical proof of this theorem the reader is referred to \textcite{momihara2017}. The author of \cite{momihara2017} mentions that, since $(p^{2n}+1)(p^n-1)$ divides $p^{4n}-1$, we can use \autoref{th:Fq_DDF} to construct a disjoint difference family $C = \{C_0, C_1, \dots, C_{p^n}\}$ with the same parameters as in \autoref{th:momihara_3.1} in the additive group of the finite field $\Fqfour$. According to \autoref{prop:dev_design} the developments $dev(C)$ and $dev(D)$ of the difference families form $2$\nobreakdash-designs with parameters $(v,k,\lambda) = \left(p^{4n}, (p^{2n}+1)(p^n-1), (p^{2n}+1)(p^n-1) - 1\right)$. The comparison of these two $2$-designs leads to our first main theorem:
	
	\begin{theorem}
	\label{th:designs_noniso}
		Let $C$ be a $\left(p^{4n}, (p^{2n}+1)(p^n-1), (p^{2n}+1)(p^n-1) - 1 \right)$ disjoint difference family in the additive group of the finite field $\F_{p^{4n}}$ constructed with \autoref{th:Fq_DDF}, and let $D$ be a disjoint difference families with the same parameters in the additive group of the Galois ring $\GR(p^2,2n)$ constructed with \autoref{th:momihara_3.1}. The $2$-$\left(v,k,k-1\right)$~designs $dev(C)$ and $dev(D)$ with parameters $v = p^{4n}$ and $k = (p^{2n}+1)(p^n-1)$ are nonisomorphic.
	\end{theorem}
	
	There are various ways of isomorphism testing for combinatorial designs. One popular approach is to study the ranks of their incidence matrices. However, in our case this yields no valid results. We were more successful examining the so-called block intersection numbers of both $2$-designs. We will prove \autoref{th:designs_noniso} by showing that the block intersection numbers of $dev(C)$ and $dev(D)$ differ. The \emph{block intersection numbers} of a $t$-design are the cardinalities $\left| B_i \cap B_j \right|$ of the intersections of two distinct blocks $B_i, B_j$ of the design.
	
	\begin{remark}
	\label{remark:intersectionNumbers_calc}
		Block intersection numbers can be easily calculated in the following way: Let $M$ denote the incidence matrix of a $t$-design with the rows of $M$ corresponding to the points and the columns of $M$ corresponding to the blocks of the design. The entry $(i,j)$ of the matrix $M^T M$ is exactly $|B_i \cap B_j|$.
	\end{remark}
	
	Block intersection numbers of combinatorial designs are invariant under isomorphism. So, to prove that our designs are nonisomorphic it is sufficient to show that $dev(D)$ has one block intersection number different from the block intersection numbers of $dev(C)$.
	
	We will first calculate all the intersection numbers of $dev(C)$. These are given as the so-called cyclotomic numbers: Analogously to \autoref{sec:DDF_Fq}, let $C_0, C_1, \dots, C_{e-1}$ be the cosets of the subgroup $C_0$ of the $e$-th powers in $\Fq^*$. For fixed non-negative integers $i,j \le e-1$ the \emph{cyclotomic number $(i,j)_e$ of order $e$} is defined as
	\begin{equation*}
	\label{eq:cyclo_number}
		(i,j)_e = |(C_i + 1) \cap C_j|.
	\end{equation*}
	In general, it is a hard number theoretic problem to calculate these cyclotomic numbers. However, \textcite{baumert1982} proved that in special cases they are easy to calculate:
	
	\begin{proposition}[{\cite[Theorems~1 and~4]{baumert1982}}]
	\label{prop:uniform_cyclotomy}
		Let $p$ be a prime, and let $e \ge 3$ be a divisor of $p^m-1$ for a positive integer $m$. If $-1$ is a power of $p$ modulo $e$, then either $p = 2$ or $f = (p^m-1)/e$ is even, $p^m = s^2$ and $s \equiv 1 \pmod e$, and the cyclotomic numbers of order $e$ are given as
		\begin{align}
		\label{eq:uniform_cyclotomy}
			(0,0)_e &= \eta^2 - (e-3)\eta - 1, && \nonumber\\
			(0,i)_e = (i,0)_e = (i,i)_e &= \eta^2 + \eta &&\textrm{for } i \ne 0,\\
			(i,j)_e &= \eta^2 &&\textrm{for } i \ne j \textrm{ and } i,j \ne 0, \nonumber
		\end{align}
		where $\eta = (s-1)/e$.
	\end{proposition}
	
	Because there exist only three distinct cyclotomic numbers in the described case, the authors of \cite{baumert1982} speak of \emph{uniform} cyclotomic numbers. Applying \autoref{prop:uniform_cyclotomy} to $dev(D)$ leads to
	
	\begin{corollary}
	\label{cor:intersection_numbers_devC}
		Let $C$ be a $\left(p^{4n}, (p^{2n}+1)(p^n-1), (p^{2n}+1)(p^n-1) - 1 \right)$ disjoint difference family in the additive group of the finite field $\F_{p^{4n}}$ constructed with \autoref{th:Fq_DDF}. The $2$-$(p^{4n}, (p^{2n}+1)(p^n-1), (p^{2n}+1)(p^n-1) - 1)$ design $dev(C)$ has exactly three block intersection numbers, namely $p^n-2$, $p^n(p^n-1)$ and $(p^n-1)^2$.
	\end{corollary}
	\begin{proof}
		We show that $C$ meets the conditions of \autoref{prop:uniform_cyclotomy}: In our case $e = p^n + 1$. So $-1$ is the $n$-th power of $p$ modulo $e$. Moreover, from $p^{4n} = s^2$ it follows that $s = p^{2n} \equiv 1 \pmod e$, so $\eta = (p^{2n}-1)/(p^n+1) = p^n-1$. By \cref{eq:uniform_cyclotomy} we obtain the cyclotomic numbers
		\begin{align*}
			(0,0)_{p^n+1} &= p^n-2, \nonumber\\
			(0,i)_{p^n+1} = (i,0)_{p^n+1} = (i,i)_{p^n+1} &= p^n(p^n-1) &&\textrm{for } i \ne 0,\\
			(i,j)_{p^n+1} &= (p^n-1)^2 &&\textrm{for } i \ne j \textrm{ and } i,j \ne 0, \nonumber
		\end{align*}
		that occur as the intersection numbers of $dev(C)$.
	\end{proof}
	
	The next step will be to show that there is a block intersection number in $dev(D)$ that does not occur in $dev(C)$.
	
	\begin{lemma}
		\label{lem:intersection_numbers_devD}
		Let $D$ be a $\left(p^{4n}, (p^{2n}+1)(p^n-1), (p^{2n}+1)(p^n-1) - 1 \right)$ disjoint difference family in the additive group of $\GR(p^2,2n)$ constructed with \autoref{th:momihara_3.1}. The $2$-$(p^{4n}, (p^{2n}+1)(p^n-1), (p^{2n}+1)(p^n-1) - 1)$ design $dev(D)$ has a block intersection number $(2p^n-1)(p^n-2)$.
	\end{lemma}
	\begin{proof}
		This proof has a similar structure to the proofs by \textcite[Lemmata~4--7]{momihara2017}, but unlike \textcite{momihara2017}, we will not consider all the sets $D_0, D_1, \dots, D_b$ of the difference family, but only $D_0$. This requires a more detailed analysis of the intersection relations.
		
		Let, like above, $\xi$ be a generator of the Teichmüller group, $x_j \in S$, where $1+pS$ is a system of representatives of $\mathbb{P}_{2n}/\mathbb{P}_n$, and $R_n^*$ denote the unit group of the subring $R_n = \GR(p^2, n)$. Furthermore, define subsets $U$ and $V$ of $R_{2n}^*$ as 
		\begin{align*}
			U = \bigcup_{j=0}^{p^n-1} \xi^j (1+px_j)R_n^* \qquad \mbox{and} \qquad V = \bigcup_{j=0}^{p^n-1} (1+px_j)R_n^*.
		\end{align*}
		Note that $D_i = \xi^i(P \cup U)$ and that $V =  \T_n^* \times \mathbb{P}_{2n}$ and $\bigcup_{j=0}^{p^n} \xi^j V = R_{2n}^*$. We will prove \autoref{lem:intersection_numbers_devD} by showing that the block intersection number $| (D_0+u) \cap D_0 |$ of the block $D_0$ and its translate $D_0+u$ equals $(2p^n-1)(p^n-2)$ for all $u \in U$. The above statement is equivalent to: An arbitrary element $u \in U$ occurs exactly $(2p^n-1)(p^n-2)$ times in the multiset $\Delta D_0$. Looking at the structure of 
		\[
			D_0 = P \cup U = P \cup \left( \bigcup_{j = 0}^{p^n-1} \xi^j (1+px_j)R_n^*\right),
		\]
		we can split all the differences in this multiset into four different types. Let $0 \le s,t \le p^n-1$ and $s \ne t$. We have differences of
		
		\begin{enumerate}[label=\emph{Type \arabic*:}, ref=type~\arabic*, leftmargin=*, labelsep=2em]
			\item $\xi^s (1+px_s)R_n^* - \xi^t (1+px_t)R_n^*$, \label{enum:1}
			\item $\Delta\left(\xi^s (1+px_s)R_n^*\right)$, \label{enum:2}
			\item $\xi^s (1+px_s)R_n^* - P$, \label{enum:3}
			\item $\Delta P$. \label{enum:4}	
		\end{enumerate}
		
		Now, let $u$ be a fixed element of $U$. We will count the number of occurrences of $u$ in $\Delta D_0$ by counting its occurrences in each of the four multisets defined above. Before we start we state the following useful lemma which will show that it does not matter whether we look at the differences or the sums in \ref{enum:1}--\ref{enum:4}:
		
		\begin{lemma}
		\label{lem:-1_GR}
			Consider the Galois Ring $\GR(p^m, r)$. If $p$ is odd, then $-1$ is an element of the Teichmüller group $\T^*$. If $p = 2$, then $-1$ is a principal unit.
		\end{lemma}
		\begin{proof}
			Let $p$ be odd. Then, the group of principal units $\mathbb{P}$ is a direct product of $r$ cyclic groups, each of odd order $p^{m-1}$, and the Teichmüller group $\T^*$ has even order $p^r-1$. Consequently, there are only two second roots of unity in $\GR(p^m,r)$: $1$ and $-1$. Hence, $-1 = \xi^{(p^{r}-1)/2} \in \T^*$. If $p=2$, all the even integers of $\Z_{p^m} \subseteq \GR(p^m,r)$ are elements of the maximal ideal $\I = 2R$. Since $-1$ is odd, it follows that $-1 \in \mathbb{P} = 1 + \I$.
		\end{proof}
		
		In our case, $m=2$ and $r = 2n$, it follows from \autoref{lem:-1_GR} that if $p$ is odd, $-1 = \xi^{(p^{2n}-1)/2} = \xi^{(p^n-1)(p^n+1)/2}$, and thus $-1$ is included in the Teichmüller group $\T_n^* = \{\xi^{i(p^n+1)} : i = 0,1,\dots,p^n-2\}$ of the subring $R_n$. Because $\T_n^* \subseteq R_n^*$ and $P = p \xi^{p^n} \T_n^*$, we have 
		\begin{align}
		\label{eq:-R=R_-P=P}
			R_n^* = -R_n^*, \qquad\textrm{and}\qquad P = -P.
		\end{align}
		If $p = 2$, we have $-1 = 3 = 1\cdot(1+2\cdot1) \in R_n^*$, since clearly $1 \in \T_n^*$. Furthermore, since $-2 = 2$, all the elements of $\I_{2n} = 2R_{2n}$ (and consequently of $P \subseteq \I_{2n}$) are self inverse. Thus, \cref{eq:-R=R_-P=P} holds in this case as well.
		Now we start our proof by analyzing differences of  \ref{enum:1}, and we first state a helpful lemma \cite{momihara2017}:
		
		\begin{lemma}[{\cite[Lemma~3]{momihara2017}}]
		\label{lem:momihara_3.3}
			Let $a$ be an integer, $0 \le a \le p^n-2$, let $b$ be an element of $R_{2n}$, and let $V$ be as defined above. If $\xi^a(1+pb) \notin R_n^*$ and $\xi^a \notin \T_n$, then
			\begin{equation*}
				R_n^* + \xi^a(1+pb)R_n^* = R_{2n}^* \setminus \left( V \cup \xi^a V\right).
			\end{equation*}
		\end{lemma}
		
		Now, let $0 \le s,t \le p^n-1, s \ne t$ be fixed. By \cref{eq:-R=R_-P=P} we have
		\begin{align*}
			\xi^s (1+px_s)R_n^* - \xi^t (1+px_t)R_n^* &= \xi^s (1+px_s)R_n^* + \xi^t (1+px_t)R_n^*.
		\end{align*}
		We factor out $\xi^s(1+px_s)$ and obtain
		\begin{align*}
			\xi^s(1+px_s) \left(R_n^* + \xi^{t-s} \left(1+p(x_t-x_s)\right)R_n^*\right).
		\end{align*}
		Applying \autoref{lem:momihara_3.3}, this equals
		\begin{align*}
			\xi^s(1+px_s) \left(R_{2n}^* \setminus \left( V \cup \xi^{t-s} V\right)\right).
		\end{align*}
		Since $(1+px)V = V$ for any $x \in R_{2n}$, the factor $(1+px_s)$ can be omitted, and we obtain the result
		\begin{align}
		\label{eq:type1}
			\xi^s (1+px_s)R_n^* - \xi^t (1+px_t)R_n^* = R_{2n}^* \setminus \left( \xi^s V \cup \xi^t V \right).
		\end{align}
		Now, we are able to count differences: The set $D_0$ contains $p^n$ distinct subsets of the type $\xi^s(1+px_s)R_n^*$. Consequently, we have $p^n(p^n-1)$ \ref{enum:1} multisets. Since $0 \le s,t \le p^n-1$, the elements of $\xi^{p^n} V$ are, according to \cref{eq:type1}, contained in each of these multisets, whereas the elements of $\bigcup_{j=0}^{p^n-1} \xi^j V = R_{2n}^* \setminus \xi^{p^n}V$ occur in only $(p^n-2)(p^n-1)$ of them. Since $U$ is a subset of $R_{2n}^* \setminus \xi^{p^n}V$, we count, so far, $(p^n-2)(p^n-1)$ occurrences of our element $u$.
		
		In the next step, we will address differences of \ref{enum:2}. We will argue similarly to \textcite[Lemma~6]{momihara2017}. Let $s$ be a fixed integer, $0 \le s \le p^n-1$. We consider the set
		\begin{align*}
			\Delta\left(\xi^s (1+px_s)R_n^*\right) = \left\lbrace \xi^s(1+px_s)a - \xi^s(1+px_s)b : a,b \in R_n^*, a \ne b \right\rbrace.
		\end{align*}
		Since $a,b \in R_n^*$, we write $a = \xi^{a_1}(1+pa_2)$ and $b = \xi^{b_1}(1+pb_2)$, where $a_1, b_1 \in \{j(p^n+1) : j = 0,1,\dots,p^n-2\}$, $a_2,b_2 \in \T_n$, and $(a_1,a_2) \ne (b_1, b_2)$. We will consider two cases. First, if $a_1 = b_1$, we have
		\begin{align}
		\label{eq:a1=b1_1}
			\xi^s(1+px_s)\xi^{a_1}(1+pa_2) - \xi^s(1+px_s)\xi^{a_1}(1+pb_2).
		\end{align}
		We factor out $\xi^{s+a_1}$ and multiply the elements of $\mathbb{P}_n$:
		\begin{align*}
			\xi^{s+a_1} \left(1+p \left(x_s + a_2\right) - \left(1+p \left(x_s+b_2\right)\right)\right).
		\end{align*}
		We summarize this expression and obtain
		\begin{align*}
			\xi^{s+a_1}p(a_2-b_2),
		\end{align*}
		which is clearly an element of the maximal ideal $\I_{2n}$. Thus, the case $a_1 = b_1$ leads to no additional differences representing the unit $u$. Now, let $a_1 \ne b_1$. Instead of \cref{eq:a1=b1_1}, we now have
		\begin{align*}
			\xi^s(1+px_s)\xi^{a_1}(1+pa_2) - \xi^s(1+px_s)\xi^{b_1}(1+pb_2).
		\end{align*}
		Factoring out $\xi^s(1+px_s)$ yields
		\begin{align*}
			\xi^s(1+px_s) \left(\xi^{a_1}(1+pa_2) - \xi^{b_1}(1+pb_2)\right),
		\end{align*}
		which equals
		\begin{align*}
			\xi^s(1+px_s) \left(\xi^{a_1}-\xi^{b_1} + p(\xi^{a_1}a_2 - \xi^{b_1}b_2)\right).
		\end{align*}
		Because $\xi^{a_1}-\xi^{b_1} \ne 0$, and $\xi^{a_1}, \xi^{b_1}, a_2, b_2 \in R_n$, the terms $\xi^{a_1}-\xi^{b_1} + p(\xi^{a_1}a_2 - \xi^{b_1}b_2)$ represent the elements of $R_n^*=R_n \setminus \I_n$. The unit group $R_n^*$ contains $p^n(p^n-1)$ elements, and we can choose $a_1, a_2, b_1, b_2$ in $p^{2n}(p^n-1)(p^n-2)$ different ways. The differences are evenly distributed, which means that our element $u$ has $p^n(p^n-2)$ distinct representations of \ref{enum:2}.
		
		By adding our numbers of \ref{enum:1} and \ref{enum:2} difference representations of $u$ we have already reached the number $(2p^n-1)(p^n-2)$ stated in \autoref{lem:intersection_numbers_devD}. So, for the remaining types, we need to show that $u$ does not occur in multisets of of \ref{enum:3} and \ref{enum:4}. We examine differences of \ref{enum:3}: $\xi^s (1+px_s)R_n^* - P$. First, we take arbitrary elements $\xi^{k(p^n+1)}(1+pa), a \in \T_n$, from $R_n^*$ and $-p\xi^{\ell(p^n+1)+p^n}$ from $P$ (recall that $P = -P$), where $k,\ell \in \{0,1,\dots,p^n-2\}$ and $a \in \T_n$. So, we are interested in differences of the form
		\begin{align*}
			\xi^s (1+px_s) \xi^{k(p^n+1)} (1+pa) + p\xi^{\ell(p^n+1)+p^n},
		\end{align*}
		where $s \in \{0,1, \dots, p^n-1\}$ and $x_s \in S$ are fixed. We factor out $\xi^{s+k(p^n+1)}$, summarize, and obtain
		\begin{align}
		\label{eq:type3_3}
			\xi^{s+k(p^n+1)} \left( 1 + px_s \right) \left( 1 + pa \right) \left( 1 + p\xi^{(\ell-k)(p^n+1)+p^n-s} \right).
		\end{align}
		We write \cref{eq:type3_3} with respect to all $0 \le k, \ell \le p^n-2$ and all $a \in \T_n$:
		\begin{align}
		\label{eq:type3_4}
			\xi^s (1+px_s) \left(\T_n^* \mathbb{P}_n \left( 1+p\xi^{p^n-s}\T_n^*\right) \right).
		\end{align}
		Since $\xi^{p^n-s} \notin \T_n$, it is clear that $\mathbb{P}_n$ and $1+p\xi^{p^n-s}\T_n^*$ are disjoint. We show that each of the $p^n-1$ other cosets in $\mathbb{P}_{2n} / \mathbb{P}_n$ is represented exactly once by $1+p\xi^{p^n-s}\T_n^*$. This is equivalent to showing that $p\xi^{p^n-s}\T_n$ represents every coset of $\I_{2n}/\I_n$ except $\I_n = p\T_n^*$ itself. Assume there were two elements from the same coset in $p\xi^{p^n-s}\T_n^*$, then their difference would be in $\I_n$. However, for two distinct integers $k,\ell$, the difference
		\begin{align*}
			p\xi^{p^n-s+k(p^n+1)} - \left(p\xi^{p^n-s+\ell(p^n+1)}\right) = \xi^{p^n-s}\left(p\xi^{k(p^n+1)}-p\xi^{\ell(p^n+1)}\right)
		\end{align*}
		is not in $\I_n$ since $\xi^{p^n-s} \notin \T_n$ and $\left(p\xi^{k(p^n+1)}-p\xi^{\ell(p^n+1)}\right) \in \I_n$.
		It follows that \cref{eq:type3_4} equals
		\begin{align*}
			\xi^s (1+px_s) \left(\T_n^*  \left(\mathbb{P}_{2n} \setminus \mathbb{P}_n \right) \right).
		\end{align*}
		With the definition of $V = \T_n^* \times \mathbb{P}_{2n}$ from above, we have
		\begin{align*}
			\xi^s (1+px_s)R_n^* - P = \xi^s \left(V \setminus (1+px_s)R_n^*\right).
		\end{align*}
		Recall that $U = \bigcup_{j=0}^{p^n-1} \xi^j (1+px_j)R_n^*$. Hence $u \in U$ does not occur in multisets of \ref{enum:3}. We finish our proof by examining differences of \ref{enum:4}: Since $P$ is a subset of the maximal ideal $\I_{2n}$ the set of all the differences of distinct elements of $P$ is also a subset of $\I_{2n}$. Thus, $\Delta P$ yields no representations of $u$.
	\end{proof}
	
	For $p^n > 2$, the intersection number $(2p^n-1)(p^n-2)$ of $dev(D)$ does not equal any of the cyclotomic numbers $p^n-2$, $p^n(p^n-1)$ and $(p^n-1)^2$ of $dev(C)$. Hence, \autoref{cor:intersection_numbers_devC} in combination with \autoref{lem:intersection_numbers_devD} proves \autoref{th:designs_noniso} for $p^n > 2$. In the case $p^n=2$, i.\,e., $p=2$ and $n=1$, however, the intersection numbers of $dev(C)$ match those from $dev(D)$. We complete the proof of \autoref{th:designs_noniso} by using the computer algebra system \texttt{Magma}~\cite{magma} to compute the full automorphism groups of $dev(C)$ and $dev(D)$. We see that, for $p^n=2$, the automorphism group of $dev(C)$ has order~$960$, whereas for $dev(D)$ it is only of order~$192$. Thus, these $2$-$(16,5,4)$~designs are nonisomorphic as well.
	
	\begin{remark}
	\label{remark:automorphism_group}
		The full automorphism group of $dev(D)$ has order $4np^{4n}(p^{4n}-1)$, it consists of the additive group $(\Fqfour, +)$ of order $p^{4n}$, the multiplicative group $\Fqfour^*$ of order $p^{4n}-1$ and the Galois group $\mathrm{Gal}(\Fqfour / \F_p)$ of order $4n$. The full automorphism group of $dev(C)$ has order $2p^{5n(p^{2n}-1)}$, it consists of the additive group $(R_{2n}, +)$ of order $p^{4n}$, the Teichmüller group $\T_{2n}$ of $R_{2n}$ of order $p^{2n}-1$, the group of principal units $\mathbb{P}_n$ of the subring $R_n$ of order $p^n$ and an interesting automorphism of order~$2$ of the additive group $(R_{2n}, +)$. In the case $p^n = 2$, it is defined by $1 \mapsto 1 + 2\xi, \xi \mapsto 1 + 3\xi$.
	\end{remark}

\section{Disjoint difference families in Galois rings II}
	\label{sec:EDF_GR}
	\textcite{davis2017} found a new cyclotomic construction of near-complete $(v,k,k-1)$ external difference families in Galois rings $\GR(p^2,r)$ of characteristic $p^2$. This construction was in a more general way already given by \textcite{furino1991} in 1991 for arbitrary commutative rings with an identity. \textcite{furino1991} used the approach to create near-complete disjoint difference families, and we know from \autoref{prop:EDF_DDF} that every near-complete disjoint difference family is also a near-complete external difference family. Furthermore, we remark that the construction by \textcite{furino1991} was generalized in the case $(v,k,k-1)$ by \textcite{buratti2017} to so-called \emph{Ferrero pairs}~$(G,A)$, where $A$ is a non-trivial group of automorphisms of $G$ acting semiregularly on a group $G\setminus\{0\}$. Before we state the result by \textcite{davis2017} we need the following two useful lemmas about differences in Galois rings. Let, like before, $\T$ denote the Teichmüller set, $\T^* = \T \setminus \{0\}$ denote the cyclic Teichmüller group having order $p^r-1$, and $\I = p\GR(p^2,r)$ denote the maximal ideal of the Galois ring $\GR(p^2,r)$.
	\begin{lemma}
	\label{lem:unit_difference_unit}
		In $\GR(p^2, r)$ the difference $u-u'$ of two distinct units $u = \alpha_0(1+p\alpha_1), u' = \alpha_0'(1+p\alpha_1')$, where $\alpha_0, \alpha_0' \in \T^*, \alpha_1, \alpha_1' \in \T$, is a unit if and only if $\alpha_0 \ne \alpha_0'$.
	\end{lemma}
	\begin{proof}
		Let $u = \alpha_0(1+p\alpha_1), u' = \alpha_0'(1+p\alpha_1')$, where $\alpha_0, \alpha_0' \in \T^*, \alpha_1, \alpha_1' \in \T$. If $\alpha_0 = \alpha_0' = \alpha$, we have
		\begin{align*}
			u-u' = \alpha(1+p\alpha_1) - \alpha(1+p\alpha_1') = p\alpha(\alpha_1-\alpha_1'),
		\end{align*}
		which is an element of $\I$. If $\alpha_0 \ne \alpha_0'$, we have
		\begin{align*}
			u-u' = \alpha_0(1+p\alpha_1) - \alpha_0'(1+p\alpha_1') = \alpha_0 - \alpha_0' + p(\alpha_1-\alpha_1'),
		\end{align*}
		which is clearly a unit since $\T$ is a system of representatives of $\GR(p^2,r) / \I$.
	\end{proof}
	
	\begin{lemma}
	\label{lem:differences_in_T_cosets}
		\begin{enumerate}
			\item The multiset $\Delta \T^*$ contains only units of $\GR(p^2, r)$. 
			\item Let $d \in (1+p\beta)\T^*$ for some $\beta \in \T^*$. If $d$ is contained in $\Delta \T^*$ then the whole coset $(1+p\beta)\T^*$ is a subset of $\Delta \T^*$.
		\end{enumerate}
	\end{lemma}
	\begin{proof}
		The first statement follows immediately from \autoref{lem:unit_difference_unit}. The second statement can be proved as follows: Let $d \in \Delta \T^*$. Then $d\T^*$ is a multiplicative coset of $\T^*$, and there exists $\beta \in \T^*$ such that $(1+p\beta)\T^* = d\T^*$. Since $d \in \Delta \T^*$, there are distinct elements $\alpha, \alpha' \in \T^*$ such that $d = \alpha-\alpha'$. Since for every $\gamma \in \T^*$ the elements $\alpha\gamma, \alpha'\gamma$ are contained in $\T^*$, the set $d\T^* = \{d\gamma : \gamma \in \T^*\}$ is a subset of $\Delta \T^*$.
	\end{proof}	
	Let us now present the construction of $(v,k,k-1)$ disjoint difference families by \textcite{davis2017}. Since the authors of \cite{davis2017} proved the result in terms of external difference families, we will include a short proof that is analogous to the proof of \autoref{th:Fq_DDF}. We remark that the theorem can also be proved using the results by \textcite{furino1991} or the Ferrero pairs by \textcite{buratti2017}.
	
	\begin{theorem}[{\cite[Theorem~4.1]{davis2017}}]
	\label{th:Davis_EDF}
		Let $\T$ be the Teichmüller set of the Galois ring $\GR(p^2,r)$, and let $\T^* = \T \setminus \{0\}$. The collection
		\begin{equation*}
		\label{eq:Davis_EDF}
			E = \left \lbrace (1+p\alpha)\T^* : \alpha \in \T \right \rbrace \cup p\T^*
		\end{equation*}
		forms a near-complete $(p^{2r}, p^r-1, p^r-2)$ disjoint difference family in the additive group of $\GR(p^2, r)$.
	\end{theorem}
	\begin{proof}
		We will first count the number of differences for the units of $\GR(p^2,r)$ and next for the non-invertible elements. Let $x,y \in \GR(p^2,r)^*$. Assume $x = u-u'$, where $u,u'$ are elements of the same coset $(1+p\alpha)\T^*$ of the Teichmüller group $\T^*$ for some fixed element $\alpha \in \T$. There exists $z \in \GR(p^2,r)$ so that $y = zx$. Hence, $y = zu-zu'$ and $zu, zu' \in (1+p\alpha')\T^*$ for some $\alpha' \in \T$, and we have found a representation of $y$ as the difference of two distinct elements from the same set $(1+p\alpha')\T^*$. Since every difference for $y$ will also give us a difference for $x$, it follows that every unit will have the same number of differences. In each of the $p^r$ sets $(1+p\alpha)\T$ we have $(p^r-1)(p^r-2)$ differences, giving us a total of $p^r(p^r-1)(p^r-2)$ differences. From \autoref{lem:differences_in_T_cosets}, we know that all these differences are units. Since there are $p^r(p^r-1)$ units in $\GR(p^2,r)$, each unit has
		\begin{align*}
			\frac{p^r(p^r-1)(p^r-2)}{p^r(p^r-1)} = p^r-2
		\end{align*}
		representations as a difference from two distinct elements of the sets $(1+p\alpha)\T^*, \alpha \in \T$.\par
		We now consider the non-invertible non-zero elements of $\GR(p^2,r)$, i.\,e. the elements of the set $p\T^* = \I \setminus \{0\}$. Since $\I$ is a group under addition, $\I \setminus \{0\}$ is a trivial $(p^r,p^r-1,p^r-2)$ difference set in $\I$. Hence, $\Delta p\T^* = (p^r-2) (\I \setminus \{0\})$. Combining these two results, we see that every non-zero element of $\GR(p^2,r)$ has $p^r-2$ differences in the sets of $E$.
	\end{proof}
	
	\textcite{davis2017} remark that, since $p^r+1$ divides $p^{2r}-1$, there is also a disjoint difference family with the same parameters in $(\F_{p^{2r}},+)$ which can be constructed using \autoref{th:Fq_DDF}. The authors ask whether the associated designs, i.\,e. the developments of these two disjoint difference families, are isomorphic. We will answer this question by showing that the designs are nonisomorphic in all but one case. This is our second main theorem.
	
	\begin{theorem}
	\label{th:designs_noniso_Davis}
		Let $C$ be a $(p^{2r}, p^r-1, p^r-2)$ disjoint difference family in the additive group of the finite field $(\F_{p^{2r}},+)$ constructed with \autoref{th:Fq_DDF}, and let $E$ be a disjoint difference family with the same parameters in the additive group of the Galois ring $(\GR(p^2,r),+)$ constructed with \autoref{th:Davis_EDF}. The $2$-$(p^{2r}, p^r-1, p^r-2)$ designs $dev(C)$ and $dev(E)$ are isomorphic if $p=3$ and $r=1$, and they are nonisomorphic in every other case.
	\end{theorem}
	
	To prove \autoref{th:designs_noniso_Davis} we will consider four cases: First, we will examine the case $p=3$ and $r=1$, second, we will consider $p=2$ and $r=2$, third, we look at the case $p=2$ and $r \ge 3$, and last we will consider $p \ge 3$ and arbitrary $r$ (except $p=3$ and $r=1$).
	
	If $p=3$ and $r=1$, the $2$-designs $dev(C)$ and $dev(E)$ are isomorphic. In this case $\GR(9,1) \cong \Z_9$ and $(\F_9,+) \cong \Z_3 \times \Z_3$. An isomorphism between $dev(E)$ and $dev(C)$ computed by \texttt{Magma}~\cite{magma} is the map $f:\Z_9 \to \Z_3 \times \Z_3$ on the point set of $dev(E)$ with
		\begin{align*}
			0 &\mapsto (0,0),&& 1 \mapsto (0,1),&& 2 \mapsto (1,2),&& 3 \mapsto (1,1), && 4 \mapsto (2,2)\\
			5 &\mapsto (2,0),&& 6 \mapsto (1,0),&& 7 \mapsto (2,1),&& 8 \mapsto (0,2).
		\end{align*}
	
	If $p=2$ and $r=2$, the designs $dev(E)$ and $dev(C)$ share the same block intersection numbers (together with their multiplicities), which we use in the proof of the remaining two cases. For both designs the block intersection numbers are $0$~($1600$~times), $1$~($1440$ times) and $2$ ($120$ times). Hence, we solve this case by computing the automorphism groups of the designs with the help of \texttt{Magma}~\cite{magma}: The automorphism group of $dev(E)$ has order $384$ while the automorphism group of $dev(C)$ is of order $5760$. If $dev(E)$ and $dev(C)$ were isomorphic, their automorphism groups would be the same. Hence, the two designs are nonisomorphic.
	
	For the case $p=2$ and $r \ge 3$ and the case $p\ge 3$ (except $p=3$ and $r=1$) we will prove \autoref{th:designs_noniso_Davis} similarly to \autoref{th:designs_noniso} by showing that the block intersection numbers of $dev(C)$ and $dev(E)$ differ. We start by calculating the block intersection numbers of $dev(C)$, that are the cyclotomic numbers of order $p^r+1$ in $\F_{p^{2r}}$. As before, these cyclotomic numbers are uniform (see \autoref{prop:uniform_cyclotomy}):
	
	\begin{corollary}
	\label{cor:intersection_numbers_devC_Davis}
		Let $C$ be a $(p^{2r}, p^r-1, p^r-2)$ disjoint difference family in the additive group of the finite field $\F_{p^{2r}}$ constructed with \autoref{th:Fq_DDF}. The $2$-$(p^{2r}, p^r-1, p^r-2)$ design $dev(C)$ has exactly three block intersection numbers, namely $p^r-2$, $0$ and $1$.
	\end{corollary}
	\begin{proof}
		Since $e = p^r+1$, we have $-1 \equiv p^r \pmod e$, and we are allowed to employ \autoref{prop:uniform_cyclotomy}. From $s^2 = p^{2r}$ and $s \equiv 1 \pmod e$ it follows that $s = -p^r$. Thus, $\eta = (-p^r-1)/(p^r+1) = -1$, and we get the cyclotomic numbers
		\begin{align*}
			(0,0)_{p^r+1} &= p^r-2, \nonumber\\
			(0,i)_{p^r+1} = (i,0)_{p^r+1} = (i,i)_{p^r+1} &= 0 &&\textrm{for } i \ne 0,\\
			(i,j)_{p^r+1} &= 1 &&\textrm{for } i \ne j \textrm{ and } i,j \ne 0, \nonumber
		\end{align*}
		that occur as the intersection numbers of $dev(C)$.
	\end{proof}
	
	\begin{remark}
	\label{rem:Fq_subfield}
		Since $r$ divides $2r$, the finite field $\F_{p^{2r}}$ contains a unique subfield $\F_{p^r}$. Thus, the subgroup $C_0 = \{ \alpha^{i(p^r+1)} : i = 0, 1, \dots, p^r-2\}$ of order $p^r-1$ of $\F_{p^{2r}}^*$ that generates the disjoint difference family $C$ is the multiplicative group $\F_{p^r}^*$ of $\F_{p^r}$. Hence, it is clear that $\Delta C_0 = (p^r-2)C_0$.
	\end{remark}
	
	We now focus on the intersection numbers of $dev(E)$ in the Galois ring $\GR(p^2, r)$. The following \autoref{lem:int_numbers_Davis} in combination with \autoref{cor:intersection_numbers_devC_Davis} will finish the proof of \autoref{th:designs_noniso_Davis}.
	\begin{lemma}
	\label{lem:int_numbers_Davis}
		Let $E$ be a $(p^{2r}, p^r-1, p^r-2)$ disjoint difference family in the additive group of the Galois ring $\GR(p^2,r)$ constructed with \autoref{th:Davis_EDF}. If $p=2$ and $r \ge 2$, the $2$-$(2^{2r}, 2^r-1, 2^r-2)$ design $dev(E)$ has a block intersection number $2$. If $p\ge 3$ (except the case $p=3$ and $r=1$), the $2$-$(p^{2r}, p^r-1, p^r-2)$ design $dev(E)$ has a block intersection number $N$ with $1 < N < p^r-2$.
	\end{lemma}
	We prove \autoref{lem:int_numbers_Davis} by a series of lemmas. We start by considering the Galois ring $\GR(4,r)$, $r \ge 2$, i.\,e., we set $p=2$, and show that, in this case, $2$ is a block intersection number of $dev(E)$ if $r \ge 2$.  Recall that in the case $p=2$, according to \autoref{lem:-1_GR}, $-1$ is a principal unit. Looking at the construction of our disjoint difference family $E$ we notice that for every block $(1+2\alpha)\T^*, \alpha \in \T$, of $E$ its additive inverse $-(1+2\alpha)\T^*, \alpha \in \T$, is also a block of $E$, and the two sets are disjoint. 
	We also need the following result:
	
	\begin{lemma}
		In the Galois ring $\GR(4,r)$, $r \ge 2$, the Teichmüller set $\T$ is the set of all squares.
	\end{lemma}
	\begin{proof}
		Let $x \in \GR(4,r)$. If $x \in \I$, say $x = p\alpha$ for some $\alpha \in \T$, then $x^2 = p^2 \alpha^2 = 0$. If $x$ is a unit, say $x = \alpha_0(1+p\alpha_1), \alpha_0 \in \T^*, \alpha_1 \in \T, \alpha_0 \ne 0$, then $x^2 = \alpha_0^2 (1+p\alpha_1)^2 = \alpha_0$ because each principal unit has order $2$.
	\end{proof}
	
	To examine the block intersection numbers of $dev(E)$ we use the well-known result that in a Galois ring of characteristic~$4$ the Teichmüller set is a relative difference set. We add a proof in \autoref{prop:TeichmuellerSet_PDS_even}. Let $G$ be a group of order $mn$ that contains a normal subgroup $N$ of order $n$. A $k$-subset $D$ of $G$ is called an \emph{$(m,n,k,\lambda)$-relative difference set} in $G$ relative to $N$ if each element of $G\setminus N$ occurs exactly $\lambda$ times and the elements of $N$ do not occur in $\Delta D$.
	
	\begin{proposition}
	\label{prop:TeichmuellerSet_PDS_even}
		In $\GR(4,r)$, $r \ge 2$, the Teichmüller set $\T = \{0,1, \xi, \dots, \xi^{2^r-2}\}$ is a $(2^r,2^r,2^r,1)$-relative difference set in the additive group of $\GR(4,r)$ relative to the maximal ideal $\I = 2\GR(4,r)$.
	\end{proposition}
	\begin{proof}
		This proof is similar to the one by \textcite[Lemmas~2 and~3]{bonnecaze1997}. From \autoref{lem:unit_difference_unit} we know that the difference $\beta-\beta'$ of two Teichmüller elements $\beta, \beta' \in \T$ is a unit if and only if $\beta \ne \beta'$. Additionally, we know that each element $x$ of $\GR(4,r)$ has a unique $2$-adic representation $x = \alpha_0 + 2\alpha_1,\ \alpha_0, \alpha_1 \in \T$. We consider the equation
		\begin{align}
		\label{eq:PDS_1}
			\alpha_0 + 2\alpha_1 = \beta - \beta'.
		\end{align}
		We choose $\beta$ and $\beta'$ and fix thereby $\alpha_0$ and $\alpha_1$. Hence, \cref{eq:PDS_1} has $4^r$ solutions $(\alpha_0, \alpha_1, \beta, \beta')$. In the next step we consider the system of equations
		\begin{align*}
			\alpha_0^2+2\alpha_0\alpha_1 + \alpha_1^2 &= \alpha_0\beta\\
			\alpha_1^2 &= \alpha_0 \beta'\\
			\beta &= \beta',\qquad \text{if } \alpha_0 = 0.
		\end{align*}
		The system of equations has also $4^r$ solutions $(\alpha_0, \alpha_1, \beta, \beta')$: It has $2^r$ solutions if $\alpha_0 = 0$, namely $(0,0,\beta, \beta)$ for arbitrary $\beta$, and it has $2^r(2^r-1)$ solutions if $\alpha_0 \ne 0$: In this case, we can chose $\alpha_0 \in \T^*, \alpha \in \T$ arbitrarily and obtain unique solutions for $\beta, \beta'$, namely 
		\begin{align}
		\label{eq:beta}
			\beta &= \alpha_0^{-1}(\alpha_0 + \alpha_1)^2,\\
		\label{eq:beta'}
			\beta' &= \alpha_0^{-1}\alpha_1^2.
		\end{align}
		It is easy to check that the solutions to the system of equations also solve \cref{eq:PDS_1}. Hence, each unit $\alpha_0 + 2\alpha_1, \alpha_0 \ne 0,$ can be uniquely represented as the difference of two Teichmüller elements $\beta, \beta'$ as described in \cref{eq:beta} and \cref{eq:beta'}.
	\end{proof}
	
	\autoref{prop:TeichmuellerSet_PDS_even} implies
	
	\begin{corollary}
	\label{cor:diff_Teichmueller_even}
		In $\GR(4, r)$, $r \ge 2$, we have for each $\alpha \in \T$ that
		\begin{align*}
			\Delta(1+2\alpha)\T^* = \GR(4,r) \setminus (\I \cup (1+2\alpha)\T^* \cup -(1+2\alpha)\T^*),
		\end{align*}
		and each element of $\GR(4,r) \setminus (\I \cup (1+2\alpha)\T^* \cup -(1+2\alpha\T^*)$ has multiplicity $1$ in this multiset.
	\end{corollary}
	\begin{proof}
		According to \autoref{prop:TeichmuellerSet_PDS_even} the Teichmüller set $\T$ is a $(2^r, 2^r, 2^r, 1)$ relative difference set in the additive group of $\GR(4,r)$ relative to $\I$. Hence, by removing $0$ from $\T$ to obtain $\T^*$ we remove differences of the type $\T^*-0 = \T^*$ and $0-\T^* = -\T^*$ from our set of differences.
	\end{proof}
	
	These results enable us to determine an intersection number of $dev(E)$. For our purpose the following \autoref{lem:sum_Teichmueller_even_atmost2} suffices. A more detailed analysis of the differences and sums in the Teichmüller set in $\GR(4,r)$, that includes the following result in a slightly different way, is given by \cite[Section~III.--C.]{hammons1994} and \textcite[Theorem~1]{bonnecaze1997}. Arguments of this type have also been used by \textcite{ghinelli2003,deresmini2002} in the theory of difference sets to obtain ovals in the development of a difference set and by \textcite{pottzhou2017} to construct Cayley graphs .
	\begin{lemma}
	\label{lem:sum_Teichmueller_even_atmost2}
		In $\GR(4,r)$, $r \ge 2$, an element $s$ of the multiset $\Delta_+(1+2\alpha)\T^*$ has multiplicity~$2$ if $s$ is a unit and multiplicity $1$ if $s \in \I \setminus \{0\}$.
	\end{lemma}
	\begin{proof}
		Let $\beta, \gamma$ be two distinct elements of the Teichmüller group $\T^*$. Since $2\T = \I$, it follows that $\beta+\beta = 2\beta \ne 2\gamma = \gamma+\gamma$. Hence, the elements of $\I \setminus \{0\}$ are represented once as the sum of two elements of $\T^*$. We now consider sums of the type $\beta + \gamma, \beta \ne \gamma$. It is clear that each sum $s = \beta + \gamma$ has at least two representations: $\beta + \gamma$ and $\gamma + \beta$. To prove that there are no more than those two representations we assume that there exist elements $\beta', \gamma' \in \T^*,\ \beta', \gamma' \notin \{\beta,\gamma\}$ such that $\beta+\gamma = \beta' + \gamma'$. This is equivalent to $\beta-\beta' = \gamma'-\gamma$. However, according to \autoref{cor:diff_Teichmueller_even} all the differences of two distinct elements of $\T^*$ are distinct. Consequently, $\beta + \gamma \ne \beta' + \gamma'$.
	\end{proof}
	
	\autoref{cor:diff_Teichmueller_even} and \autoref{lem:sum_Teichmueller_even_atmost2} lead us to the following statement about the block intersection numbers in $dev(E)$:
	
	\begin{lemma}
	In $\GR(4,r)$, $r \ge 2$, let $E_\alpha = (1+2\alpha)\T^*, \alpha \in \T,$ and let $d \in \GR(4,r)$. Then
		\begin{align*}
			\left|(E_{\alpha} + d) \cap E_{\alpha}\right| &=
			\begin{cases}
				1, 	& \textrm{if } d \in \Delta E_{\alpha},\\
				0,	& \textrm{in any other case,}
			\end{cases}\\
			\left|(E_{\alpha} + d) \cap -E_{\alpha}\right| &=
			\begin{cases}
				2, 	& \textrm{if } d \in \left(\Delta_+-E_{\alpha}\right) \setminus \I,\\
				1, 	& \textrm{if } d \in \I \setminus \{0\},\\
				0,	& \textrm{in any other case.}
			\end{cases}
		\end{align*}
	\end{lemma}
	
	Since $(1+2\alpha)\T^* + d, \alpha \in \T, d \in \GR(4,r),$ and $-(1+2\alpha)\T^*$ are blocks of the design $dev(E)$, it follows that $dev(E)$ contains blocks that intersect in two elements. According to \autoref{cor:intersection_numbers_devC_Davis} however, the design $dev(C)$ has block intersection numbers $2^r-2, 0$ and $1$, and $2^r-2 > 2$ if $r \ge 3$. We conclude that the combinatorial designs $dev(C)$ and $dev(E)$ are nonisomorphic if $p=2$ and $r \ge 3$.

%

	Now, let $p$ be an odd prime.
	
	\begin{lemma}
	\label{lem:T=-T}
		In $\GR(p^2,r)$, $p$ odd, we have $\T^* = -\T^*$ .
	\end{lemma}
	\begin{proof}
		According to \autoref{lem:-1_GR}, $-1 \in \T^*$ if $p$ is odd. 
	\end{proof}
	
	Hence, the Teichmüller group consists of pairs of an element and its additive inverse and we can write
	
	\begin{equation*}
		\T^* = \left\lbrace 1, \xi, \xi^2, \dots, \xi^{(p^r-3)/2}, -1, -\xi, -\xi^2, \dots, -\xi^{(p^r-3)/2} \right\rbrace.
	\end{equation*}
	
	From this notation we can deduce 
	\begin{lemma}
	\label{lem:pairwise_differences}
		In $\GR(p^2,r)$, $p$ odd, a difference $d$ of two distinct elements of $\T^*$ occurs at least twice in $\Delta \T^*$, except if $d \in 2\T^*$, then $d$ occurs at least once in $\Delta \T^*$. Consequently, a difference $d$ has odd multiplicity in $\Delta \T^*$ if and only if $d \in 2\T^*$.
	\end{lemma}
	\begin{proof}
		Let $d = \alpha - \alpha' \in \Delta \T^*$ be the difference of two arbitrary elements of the Teichmüller group $\T^*$. According to \autoref{lem:T=-T}, the elements $-\alpha, -\alpha'$ are also in $\T^*$. Hence, $-\alpha'- (-\alpha) = d$ is another representation of $d$ as the difference of two elements from $\T^*$. However, those two representations are the same, if $\alpha' = -\alpha$. Then $d=2\alpha$, and it is not guaranteed that $d$ has more than this single representation. It could, however, happen, that there exist more distinct pairs $(\beta_1,\beta_1'), \dots, (\beta_\ell,\beta_\ell'), \beta_1, \beta_1', \dots, \beta_\ell, \beta_\ell' \in \T^* \setminus \{\alpha, \alpha'\}$ with $\beta_i - \beta_i' = -\beta_i' - (-\beta_i) = d$. Then, $d$ has $2\ell+2$ representations as a difference if $d \notin 2\T^*$ and $2\ell+1$ representations if $d \in \T^*$.
	\end{proof}
	
	\begin{corollary}
	\label{cor:int_num_>1}
		In $\GR(p^2,r)$, $p$ odd (except $p=3$ and $r=1$), let $d \in (\Delta \T^* \setminus 2\T^*)$. Then, the block intersection number $|\T^* \cap (\T^* + d)| \ge 2$.
	\end{corollary}
	\begin{proof}
		The statement follows from \autoref{lem:pairwise_differences}: Let $d = \alpha-\alpha'$ for some distinct $\alpha, \alpha' \in \T^*, \alpha' \ne -\alpha$. Then $d$ occurs at least twice in $\Delta \T^*$. In the case $p=3$ and $r=1$, the Teichmüller group contains only two elements, namely $1$ and $-1$. Hence, there are no $\alpha, \alpha' \in \T^*$ with $\alpha \ne -\alpha'$. So, we need to exclude this case.
	\end{proof}
	To finish our proof of \autoref{lem:int_numbers_Davis} we need to show that the block intersection numbers greater than $1$ from \autoref{cor:int_num_>1} are less than $p^r-2$.
	
	\begin{lemma}
	\label{lem:d_less_p^r-2}
		In $\GR(p^2,r)$, $p$ odd (except $p=3$ and $r=1$), there is no difference $d \in \Delta \T^*$ with multiplicity $p^r-2$ in $\Delta \T^*$.
	\end{lemma}
	\begin{proof}
		Assume there is an element $d \in \Delta \T^*$ with multiplicity $p^r-2$. We know from \autoref{lem:differences_in_T_cosets} that $\Delta \T^*$ is the union of whole cosets of $\T^*$ and that the elements of the same coset have the same multiplicity. Counting multiplicities, $\Delta \T^*$ contains $(p^r-1)(p^r-2)$ elements. Hence, if $d$ has multiplicity $p^r-2$, every $d'$ from the same coset as $d$ will also have multiplicity $p^r-2$. Since $\T^*$ and its cosets contain $p^r-1$ elements, this means that $\Delta \T^*$ consists of $p^r-2$ times the same coset. Obviously, $p^r-2$ is odd. From \autoref{lem:pairwise_differences} we know, that an element $d \in \T^*$ only has odd multiplicity if $d \in 2\T^*$. Thus, we have $\Delta \T^* = (p^r-2)(2\T^*)$. Now, let $\alpha$ be an arbitrary element of $\T^*$. Then $\alpha-1 \in \Delta \T^*$, and there is an element $\beta \in \T^*$ such that $\alpha-1 = 2\beta$. However, since $-1 \in \T^*$, the element $\alpha+1$ is also in $\Delta \T^*$, and we have $\alpha+1 = 2\beta +2 = 2(\beta + 1)$. Since $p\ge 3$, the element $2$ is a unit, and thus $\beta + 1 \in \T^*$. Consequently, we have $\Delta \T^* = (p^r-2)\T^*$.\par
		In other words, $\T^*$ needs to be an (additive) $\left(p^{r-1}, p^{r-2}, p^{r-2}\right)$ difference set in $\T = \T^* \cup \{0\}$. This is only the case if $\T$ forms an additive group. It is clear, however, that this is not the case: Since $1 \in \T$, this would mean that $p\in \T$, but $p$ is not a unit. Hence, there is no $d \in \T^*$ with multiplicity $p^r-2$.
	\end{proof}
	\autoref{lem:d_less_p^r-2} immediately gives us the following result.
	\begin{corollary}
	\label{cor:int_num_<p^r-2}
		In $\GR(p^2,r)$, $p$ odd (except $p=3$ and $r=1$), the block intersection number $|\T^* \cap (\T^* + d)| < p^r-2$ for all $d \in \Delta \T^*$.
	\end{corollary}
	By combining \autoref{cor:int_num_>1} and \autoref{cor:int_num_<p^r-2} we see that for odd $p$ (except $p=3$ and $r=1$) there exists an element $d \in \GR(p^2,r)$ such that the intersection number of the blocks $\T^*$ and $\T^* + d$ of $dev(E)$ is greater than $1$ and less than $p^r-2$. This concludes the proof of \autoref{lem:int_numbers_Davis}.
	
\section{Conclusion}
	In this paper, we solve the isomorphism problem for two pairs of near-complete $(v,k,k-1)$ disjoint difference families in Galois rings and finite fields. However, there exist many more constructions of difference families, and it is a natural question to ask whether their associated designs are nonisomorphic. Hence, we leave to future work the task to solve the isomorphism problem for more disjoint difference families.\par
	Moreover, it would be nice to have a general powerful construction of difference families for which one can show that (almost) all their candidates are nonisomorphic. For parameters $(v,k,k-1)$, the construction presented by \textcite{buratti2017} seems to be powerful, and it will be interesting to check if this construction leads to nonisomorphic designs.

	\printbibliography
	
\end{document}